\documentclass[12pt]{amsart}

\usepackage{amsmath,amsfonts,amssymb,amsthm, mathtools}

\usepackage{hyperref}

\usepackage{enumitem}
\usepackage{booktabs}

\newcommand{\set}[1]{\left\{#1\right\}}
\newcommand{\fromto}[2]{#1, \dotsc, #2}
\newcommand{\setfromto}[2]{\set{\fromto{#1}{#2}}}

\usepackage{adjustbox}

\def\PP{{\mathbb P}}

\def\kk{{\Bbbk}}

\newcommand{\GL}{\mathbf{GL}}
\newcommand{\Gl}{\GL}
\newcommand{\SL}{\mathbf{SL}}

\DeclareMathOperator{\codim}{codim}
\DeclareMathOperator{\End}{End}

\DeclareMathOperator{\Gr}{Gr}

\DeclareMathOperator{\Id}{Id}

\DeclareMathOperator{\Discr}{Discr}

\newcommand{\hook}{\ensuremath{\mathbin{\text{\raisebox{.4ex}{%
  \vrule height .5pt width 1ex depth 0pt%
  \vrule height 0.8ex width .5pt depth 0pt%
}}\mathchoice{}{}{\mkern3mu}{\mkern3mu}}}}%

\title{On the locus of points of high rank}

\author[J. Buczy{\'n}ski]{Jaros{\l}aw Buczy{\'n}ski}
\address{Jaros\l{}aw Buczy\'nski,
Institute of Mathematics of the Polish Academy of Sciences,
  ul.\ \'Sniadeckich 8,
  00-656 Warszawa, Poland,
and
Faculty of Mathematics, Computer Science and Mechanics,
University of Warsaw,
ul.\ Banacha 2,
02-097 Warszawa,
Poland
 }
 \email{jabu@mimuw.edu.pl}

\author[K. Han]{Kangjin Han}
\address{Kangjin Han,
School of Undergraduate Studies,
Daegu-Gyeongbuk Institute of Science \& Technology (DGIST),
333 Techno jungang-daero, Hyeonpung-myeon, Dalseong-gun
Daegu 42988,
Republic of Korea}
\email{kjhan@dgist.ac.kr}

\author{Massimiliano Mella}
\address{Massimiliano Mella,
Dipartimento di Matematica e Informatica,
Universit\`a di Ferrara,
Via Machiavelli 35,
44100 Ferrara Italia}
\email{mll@unife.it}

\author[Z. Teitler]{Zach Teitler}
\address{Zach Teitler, Department of Mathematics, Boise State University, 1910 University Drive, Boise, ID 83725-1555, USA}
\email{zteitler@boisestate.edu}

\date{March 8, 2017}

\newcommand{\bbP}{\mathbb{P}}
\newcommand{\bbk}{\Bbbk}

\newcommand{\transpose}{\top}
\DeclareMathOperator{\rank}{rank}
\DeclareMathOperator{\Vertex}{Vertex}
\DeclareMathOperator{\pr}{pr} 
\DeclareMathOperator{\Seg}{Seg} 

\newtheorem{theorem}{Theorem}
\newtheorem{lemma}[theorem]{Lemma}
\newtheorem{proposition}[theorem]{Proposition}
\newtheorem{corollary}[theorem]{Corollary}

\theoremstyle{definition}
\newtheorem{definition}[theorem]{Definition}

\theoremstyle{remark}
\newtheorem{remark}[theorem]{Remark}
\newtheorem{example}[theorem]{Example}

\begin{document}

\begin{abstract}
  Given a closed subvariety $X$ in a projective space, the rank with respect to $X$ of a point $p$ in this projective space 
    is the least integer $r$  such that $p$ lies in the linear span of some $r$ points of $X$.
  Let $W_k$ be the closure of the set of points of rank with respect to $X$ equal to $k$.
  For small values of $k$ such loci are called secant varieties. 
  This article studies the loci $W_k$ for values of $k$ larger than the generic rank.
  We show they are nested, we bound their dimensions, and we estimate the maximal possible rank with respect to $X$ in special cases,
    including when $X$ is a homogeneous space or a curve. 
  The theory is illustrated by numerous examples, including Veronese varieties, the Segre product of dimensions $(1,3,3)$, and curves. 
  An intermediate result provides a lower bound on the dimension of any $\GL_n$ orbit of a homogeneous form. 
\end{abstract}

\maketitle

\section{Introduction}

A general $m \times n$ matrix has rank $\min\{m,n\}$, and this is the greatest possible rank.
The locus of matrices of rank at most $r$, for $r \leq \min\{m,n\}$, is well-studied:
its defining equations are well-known, along with its codimension, singularities, and so on.

Also well-studied are the loci of tensors of a fixed format and of rank at most $r$.
These, up to closure, are \emph{secant varieties} of Segre varieties.
Despite intense study, defining equations and dimensions of such secant varieties are known only in limited cases,
to say nothing of their singularities.
For introductory overviews of this, see for example \cite{MR3213518,MR2865915}.
In contrast to the matrix case, however, special tensors may have ranks strictly greater than the rank of a general tensor.
The locus of tensors with ranks greater than the generic rank is quite mysterious.
In general it is not known
what is the dimension of this locus, what are its equations, whether it is irreducible---or even whether it is nonempty.

Similarly, the closure of the locus of symmetric tensors of rank at most $r$ is a secant variety of a Veronese variety.
In this case, the dimensions of all such secant varieties are known, although the equations are not known.
The same sources \cite{MR3213518,MR2865915} also give introductions to this case as well.
But once again, special symmetric tensors may have ranks strictly greater than the rank of a general symmetric tensor.
And once again, the locus of such symmetric tensors is almost completely unknown.

Here we study high rank loci for tensors and symmetric tensors,
and for more general notions of rank.
We consider rank with respect to a nondegenerate, irreducible projective variety $X \subseteq \bbP^N$
  over an algebraically closed field $\kk$ of characteristic zero.
Let $\rank = \rank_X$ denote rank with respect to $X$, the function that assigns to each point $p \in\bbP^N$
the least number $r$ such that $p$ lies in the linear span of some $r$ points of $X$.
See \S\ref{sect: background} for more details.
For $k \geq 1$ let
\[
  W_k = \overline{\rank^{-1}(k)} = \overline{\{p \in \bbP^N : \rank(p) = k\}} .
\]
Let $g$ be the generic rank with respect to $X$.
Note that $W_k = \sigma_k(X)$ is the $k$th secant variety for $1 \leq k \leq g$,
in particular $W_1 = X$ and $W_{g} = \bbP^N$.
We seek to understand the high rank loci, namely, $W_k$ for $k > g$.

We give dimension bounds for the $W_k$ and we find containments and non-containments between the high rank loci and secant varieties.
Using these, we can improve previously known upper bounds for rank in the cases where $X$ is a curve 
  (Theorem~\ref{thm_dimension_bound_for_curve}) 
  or a projective homogeneous variety (Theorem~\ref{thm_dimension_bound_for_homogeneous}).
This includes Segre and Veronese varieties, corresponding to tensor rank and symmetric tensor rank.
The key result is a nesting statement, that each high rank locus $W_k$ for $k > g$ is contained in the next highest one $W_{k-1}$,
and in fact more strongly the join of $W_k$ and $X$ is contained in $W_{k-1}$, see Theorem~\ref{thm_nesting_of_join_and_loci}.

We give a lower bound for the dimension of the locus of symmetric tensors of maximal rank,
showing that, even though the maximal value of rank is unknown~(!),
there is a relatively large supply of symmetric tensors with maximal rank, see Theorem~\ref{thm_dimension_bound_for_Veronese}.
Possibly of independent interest, we give a lower bound for the dimension of the $\GL(V)$ orbit of a homogeneous form
$F \in S^d(V)$, assuming only that $F$ is concise, i.e., cannot be written using fewer variables;
other well-known results assume $F$ defines a smooth hypersurface, but we give a bound even if the hypersurface defined by $F$
is singular, reducible, or non-reduced, see Proposition~\ref{prop_dimension_of_orbit_of_a_concise_form}.

We find the dimension of the locus of $2 \times 2n \times 2n$ tensors of maximal rank, see Proposition~\ref{prop: dim of tensor wm},
and we characterize all $2 \times 4 \times 4$ tensors with greater than generic rank, see Proposition~\ref{prop_2x4x4_tensors}.

Finally, in Section~\ref{sect_curves} we let $X$ be a curve contained in a smooth quadric in $\PP^3$.
Then the generic rank is $g=2$ and the maximal rank is $m=2$ or $3$.
When $X$ is a general curve of bidegree $(2,2)$ we show that $W_3$ is a curve of degree $8$, disjoint from $X$,
with $4$ points of rank $2$ and all the rest of rank $3$.
Piene has shown that if $X$ is a general curve of bidegree $(3,3)$, then $W_3$ is empty, i.e., $m=2$.
We extend this to general curves of bidegree $(a,b)$ with $a \geq 4$ and $b \geq 1$.

\section{Background}\label{sect: background}

We work over a closed field $\bbk$ of characteristic zero.

For a finite dimensional vector space $V$,
let $\bbP V$ be the projective space of lines through the origin in $V$,
and for $q \in V$, $q \neq 0$, let $[q]$ be the corresponding point in $\bbP V$.
A \emph{variety} $X \subseteq \bbP V$ is a reduced closed subscheme.
We deal only with varieties in $\bbP V$ defined over $\bbk$.
Recall that a variety $X \subseteq \bbP V$ is \emph{nondegenerate} if $X$ is not contained in any proper linear subspace,
equivalently if $X$ linearly spans $\bbP V$.

\subsection{Ranks and secant varieties}

Let $X \subseteq \bbP V$ be a nondegenerate variety.
For $q \in V$, $q \neq 0$, the \emph{rank with respect to $X$} of $q$, denoted $\rank_X(q)$ or more simply $\rank(q)$,
is the least integer $r$ such that $q = x_1 + \dotsb + x_r$
for some $x_i \in V$ with $[x_i] \in X$ for $1 \leq i \leq r$.
Equivalently, $\rank(q)$ is the least integer $r$ such that $[q]$ lies in the span of
some $r$ distinct, reduced points in $X$.
We extend $\rank$ to $\bbP V$ by $\rank([q]) = \rank(q)$.

For example, \emph{tensor rank} is rank with respect to a Segre variety,
\emph{Waring rank} is rank with respect to a Veronese variety,
and \emph{alternating tensor rank} is rank with respect to a Grassmannian in its Pl\"ucker embedding.

The rank function is subadditive and invariant under multiplication by scalars.
In particular,
\[
  \rank(p)-\rank(q) \leq \rank(p+q) \leq \rank(p)+\rank(q) .
\]

The $r$th \emph{secant variety} of $X$, denoted $\sigma_r(X)$, is the closure of the union
of the planes spanned by $r$ distinct, reduced points in $X$.
Equivalently, $\sigma_r(X)$ is the closure of the set of points of rank at most $r$.

There is a unique value $g$, called the \emph{generic rank},
such that there is a Zariski open, dense subset of $\bbP^N$ of points with rank $g$.
The generic rank is the least value $r$ such that $\sigma_r(X) = \bbP^N$.
(The situation is more complicated over non-closed fields.
See for example \cite{MR3506500} for the real case.)

\subsection{Upper bounds for rank}

As long as $X$ is nondegenerate, we may choose a basis for $V$ consisting of points $x_i$ with $[x_i] \in X$,
and then every point in $\bbP V$ can be written as a linear combination of those basis elements.
This shows that every point in $\bbP V$ has rank at most $\dim V$.
In particular the values of rank are finite and bounded.

Let $m$ be the maximal rank with respect to an irreducible, nondegenerate variety $X$.
Recall the following well-known upper bounds.
\begin{theorem}[\cite{MR2628829}]\label{upper bound codim+1}
$m \leq \codim(X)+1$.
\end{theorem}
\begin{proof}
For any $q \notin X$, a general plane through $q$ of dimension $\codim X$ is spanned by its intersection with $X$
(see argument in \cite{MR2628829}, or \cite[Proposition 18.10]{harris:intro}),
which is reduced by Bertini's theorem.
This plane intersects $X$ in $\deg(X)$ many points; choosing a spanning subset shows $\rank(q) \leq \codim(X)+1$.
\end{proof}
This was also observed by Geramita when $X$ is a Veronese variety,
corresponding to the case of Waring rank \cite[pg.~60]{Geramita}.
It is false in the positive characteristic case, see \cite{MR2783180},
and it is false over the real numbers, see \cite{MR3506500}, \cite{MR2771116}.
(In the positive characteristic case and over the real numbers the bound is $\codim(X)+2$.)
\begin{theorem}[\cite{MR3368091}]\label{upper bound 2g}
$m \leq 2g$.
If $\sigma_{g-1}(X)$ is a hypersurface, then $m \leq 2g-1$.
\end{theorem}
\begin{proof}
A general line through $q \in \bbP^N$ is spanned by two points $x,y$ in the dense open set of points of rank $g$.
So $q$ is a linear combination of $x$ and $y$, and $\rank(q) \leq \rank(x) + \rank(y) = 2g$.
If $\sigma_{g-1}(X)$ is a hypersurface, a general line through $q$ contains a point $x$ of rank $g-1$ and a point $y$ of rank $g$.
Again $\rank(q) \leq \rank(x) + \rank(y) = 2g-1$.
\end{proof}
This bound holds over the real numbers and over closed fields in arbitrary characteristic,
see \cite{MR3368091}.
Over the real numbers this bound is sharp, see \cite[Theorem 2.10]{MR3506500}.
We show that, over a closed field $\bbk$ of characteristic zero,
it can be improved to $m \leq 2g-1$ in some cases, such as when $X$ is a curve or a homogeneous variety.
It is an open question whether $m \leq 2g-1$ for every variety $X$ over a closed field.

\subsection{Joins and vertices}

We recall some basic notions of joins and vertices of varieties in $\bbP^N$.
See \cite[\textsection4.6]{MR1724388} for more details.

\begin{definition}
The \emph{join} of two varieties $V_1, V_2 \subseteq \bbP^N$, denoted $J(V_1,V_2)$,
is the closure of the union of all lines spanned by points $p,q$ with $p \in V_1$, $q \in V_2$, and $p \neq q$.
We also use additive notation:
$V_1 + V_2 = J(V_1,V_2)$
and $kV = V + (k-1)V = V+\dotsb+V$, $k$ times.
In particular the secant variety $\sigma_k(X)$ is equal to $kX$.
\end{definition}

Note that if $X,Y$ are irreducible then so is $X+Y$.

\begin{definition}
Let $W \subset \bbP^N$ be a closed subscheme.
A point $p \in \bbP^N$ is called a \emph{vertex} of $W$ if $p+W = W$ set-theoretically.
The set of vertices of $W$ is denoted $\Vertex(W)$.
\end{definition}
It is well known that $\Vertex(W) \subseteq W$ and $\Vertex(W)$ is a linear space.

\begin{proposition}[{\cite[Proposition 1.3]{MR947474}}]
Let $X, Y$ be irreducible varieties in $\bbP^N$. Then
\begin{enumerate}
\item $X+Y = Y$ if and only if $X \subseteq \Vertex(Y)$,
\item $\dim X+Y = \dim Y + 1$ implies $X \subseteq \Vertex(X+Y)$.
\end{enumerate}
\end{proposition}

\begin{corollary}\label{dimension bound for chains of joins}
Let $W, X \subseteq \bbP^N$ be irreducible varieties with $X$ nondegenerate.
For every $k \geq 0$, either $\dim (W+kX) \geq \dim W + 2k$ or $W+kX = \bbP^N$.
\end{corollary}

\section{General Results}

\begin{theorem}\label{thm_nesting_of_join_and_loci}
Let $X \subseteq \bbP^N$ be an irreducible, nondegenerate variety.
Let $g$ be the generic rank and $m$ the maximal rank with respect to $X$.
Then for each $k$, $g+1 \leq k \leq m$, $W_k + X \subseteq W_{k-1}$.
In particular $W_m \subset W_{m-1} \subset \dotsb \subset W_{g+1} \subset W_g = \bbP^N$.
\end{theorem}
\begin{proof}
Let $W$ be an irreducible component of $W_m$.
A general point of $W + X$ has rank $m$ or $m-1$.
If the general point has rank $m$ then $W + X \subseteq W_m$.
Since $W+X$ is irreducible, it is contained in one of the irreducible components of $W_m$;
since $W \subseteq W+X$, it must be $W+X = W$.
But then $X \subseteq \Vertex(W)$, contradicting the nondegeneracy of $X$.
So $W + X \subseteq W_{m-1}$, which shows $W_m + X \subseteq W_{m-1}$.

Suppose inductively $W_{h+1} + X \subseteq W_h$, where $m > h \geq g+1$.
Let $W$ be an irreducible component of $W_h$.
A general point of $W + X$ has rank $h+1$, $h$, or $h-1$.
It cannot be $h$, or else once again $W \subseteq W+X \subseteq W_h$, $W = W+X$, and $X \subset \Vertex(W)$.
And it cannot be $h+1$, or else $W + X \subseteq W_{h+1}$,
which means $W+2X \subseteq W_{h+1} + X \subseteq W_h$ by induction.
But then $W+2X$ is contained in an irreducible component of $W_h$, which must be $W$ since $W \subseteq W+2X$.
So then $W = W+2X$ and $X \subseteq 2X \subseteq \Vertex(W)$.
Hence $W + X \subseteq W_{h-1}$, which shows $W_h + X \subseteq W_{h-1}$.
\end{proof}

\begin{remark}
In \S\ref{sect_binary_forms} and \S\ref{sect: 2x4x4 tensors} we will give examples where $W_k + X = W_{k-1}$ for $g+1 \leq k \leq m$.
It is an interesting problem to find an example where the inclusion $W_k + X \subseteq W_{k-1}$ is strict.
\end{remark}

\begin{corollary}
For $1 \leq k \leq m-g$, $\sigma_k(X) = kX \subset W_{m-k}$.
\end{corollary}

\begin{corollary}\label{non containment}
For $1 \leq k \leq g-1$, $\sigma_k(X) \not\subset W_{2g-k+1}$.
In particular if $m=2g$, then for $1 \leq k \leq g-1$, $\sigma_k(X) \not\subset W_{m-k+1}$.
\end{corollary}
\begin{proof}
If $\sigma_k(X) \subseteq W_{2g-k+1}$, then
\[
  \bbP^N = gX = kX + (g-k)X \subseteq W_{2g-k+1} + (g-k)X \subseteq W_{g+1} \subsetneq \bbP^N,
\]
a contradiction.
\end{proof}

\begin{remark}
Containments in the other direction need not hold.
For an example where $W_m \not\subset \sigma_{g-1}(X)$, see Remark \ref{remark: wm not contained in secant}.
\end{remark}

We give a sharp bound on the dimension of the high rank loci.

\begin{theorem}
Let $X$ be an irreducible variety in $\bbP^N$ and let $g$ be the generic rank with respect to $X$.
For every $k \geq 1$, $\codim W_{g+k} \geq 2k-1$.
\end{theorem}
\begin{proof}
We have $W_{g+k} + (k-1)X \subseteq W_{g+1} \neq \bbP^N$ by Theorem \ref{thm_nesting_of_join_and_loci}.
Then $N > \dim ( W_{g+k} + (k-1)X ) \geq \dim W_{g+k} + 2(k-1)$ by Corollary \ref{dimension bound for chains of joins}.
\end{proof}

See \S\ref{sect_binary_forms} for an example where $\codim W_{g+k} = 2k-1$ holds.

We can give improved upper bounds for ranks in two cases.
First, if $X$ is a curve, we can improve by $1$ the conclusions of Theorem~\ref{upper bound 2g}.
\begin{theorem}\label{thm_dimension_bound_for_curve}
Let $X$ be an irreducible nondegenerate curve in $\bbP^N$.
Let $g$ be the generic rank and $m$ the maximal rank with respect to $X$.
Then $m \leq 2g-1$.
Moreover, if in addition the last nontrivial secant variety $\sigma_{g-1} (X)$ is a hypersurface, 
   then  $m\le 2g -2$.
\end{theorem}
\begin{proof}
First recall that $X$ is nondefective, meaning that for $k \geq 1$, $\dim kX = \min\{N,2k-1\}$, 
see for example \cite[Introduction, Remark 1.6]{MR947474}.
Then $N > \dim (g-1)X = 2g-3$, and $N = \dim gX \leq 2g-1$.
Hence $N \in \{2g-1,2g-2\}$.

If $N$ is odd, $N = 2g-1$, then $\codim X = N-1 = 2g-2$.
By Theorem \ref{upper bound codim+1}, $m \leq \codim X + 1 = 2g-1$.

If $N$ is even, $N = 2g-2$, then $\dim \sigma_{g-1}(X) = 2g-3 = N-1$.
This is the case in which $\sigma_{g-1}(X)$ is a hypersurface.
By Theorem \ref{upper bound codim+1} again, $m \leq \codim X + 1 = 2g-2$.
\end{proof}

\begin{remark}
The above result fails over the reals, see \cite[Theorem 2.10]{MR3506500}.
\end{remark}

Second, if $X$ is a projective homogeneous variety in a homogeneous embedding then we can obtain the same improvement.
\begin{theorem}\label{thm_dimension_bound_for_homogeneous}
Let $G$ be a connected algebraic group, $V$ an irreducible representation of $G$,
and $X = G/P \subset \bbP V$ a projective homogeneous variety.
Let $g$ be the generic rank and $m$ the maximal rank with respect to $X$.
Then $m \leq 2g-1$.
Moreover, if in addition the last nontrivial secant variety $\sigma_{g-1} (X)$ is a hypersurface, 
   then  $m\le 2g -2$.
\end{theorem}
\begin{proof}
  $X$ is the unique closed orbit of $G$ on $\bbP V$, see for example \cite[Claim 23.52]{MR1153249}.
  Since $X$ is $G$-invariant, so is each rank locus $W_k$.
  Every $G$-invariant closed set contains $X$, in particular $X \subset W_m$.
  The asssertion $m=2g$ contradicts Corollary \ref{non containment}, thus $m \le 2g-1$.

  If in addition $\sigma_{g-1} (X)$ is a hypersurface, and $m = 2g-1$, 
    then 
    \[
      \sigma_{g-1}(X) = (g-1)X \subseteq W_m + (g-2)X,
    \]
    since $X \subset W_m$.
    Then
    \[
       \sigma_{g-1}(X) \subseteq W_{2g-1} + (g-2)X \subseteq W_{g+1} \subsetneqq \PP V.
    \]
  Therefore $W_{g+1}$ contains an irreducible component equal to $\sigma_{g-1} (X)$.
  This contradicts the definition of the rank locus $W_{g+1}$: general points in $\sigma_{g-1}(X)$
  have rank $g-1$, but general points in each component of $W_{g+1}$ have rank $g+1$.
  It follows that whenever $\sigma_{g-1} (X)$ is a hypersurface, we must have $m \le 2g-2$.
\end{proof}

The bounds in both Theorem \ref{thm_dimension_bound_for_curve} and Theorem \ref{thm_dimension_bound_for_homogeneous}
are attained when $X$ is a rational normal curve, see \S\ref{sect_binary_forms}.

\begin{example}
The maximal rank is strictly less than twice the generic rank in the following cases.
\begin{enumerate}
\item Waring rank, when $X$ is a Veronese variety.
\item Tensor rank, when $X$ is a Segre variety.
\item Alternating tensor rank, when $X$ is a Grassmannian in its Pl\"uck\-er embedding.
\item Multihomogeneous rank, also called partially symmetric tensor rank, when $X$ is a Segre-Veronese variety.
\end{enumerate}
\end{example}

\section{Veronese Varieties}

When $X = \nu_d(\bbP^{n-1}) \subset \bbP^N$ is a Veronese variety, then $\bbP^N$ is the projective space of degree $d$ homogeneous forms
in $n$ variables and $X$ corresponds to the $d$th powers.
Rank with respect to $X$ is called \emph{Waring rank}.
The Waring rank of a homogeneous form of degree $d$ is the least $r$ such that the form can be written as a sum of $r$ $d$th powers of linear forms.
For example, $xy = \frac{1}{4}(x+y)^2 - \frac{1}{4}(x-y)^2$, so $\rank(xy) \leq 2$; since $xy \neq \ell^2$, $\rank(xy) = 2$.

The main result in this section
is a lower bound for the dimension of the maximal rank locus $W_m$ with respect to any Veronese variety.
\begin{theorem}\label{thm_dimension_bound_for_Veronese}
   Suppose $X=\nu_d(\PP V) \subset \PP (S^dV)$ is the Veronese variety and that $\dim V = n \ge 3$.
   Then every irreducible component of  the rank locus $W_m$ has dimension at least $\binom{n+1}{2}-1$.
   Moreover, if $W$ is an irreducible component of $W_m$ with $\dim W = \binom{n+1}{2}-1$, 
     then $d$ is even and $W$ is the set of all $\frac{d}{2}$-th powers of quadrics.
\end{theorem}
This will be proved in \S\ref{sect: proof of dimension bound for veronese}.
The proof uses a lower bound for the dimension of the orbit of a homogeneous form under linear substitutions of variables,
which may be of independent interest, see \S\ref{sect: orbits of homogeneous forms}.
First, we review some background information on apolarity, conciseness, and generic and maximal Waring rank.
We also give a full description of the rank loci with respect to a rational normal curve,
that is, a Veronese embedding of $\bbP^1$, corresponding to Waring rank of binary forms.

\subsection{Apolarity}

Let $S = \bbk[x_1,\dotsc,x_n]$ and let $T = \bbk[\alpha_1,\dotsc,\alpha_n]$, called the dual ring of $S$.
We let $T$ act on $S$ by differentiation, with $\alpha_i$ acting as partial differentiation by $x_i$.
This is called the \emph{apolarity action} and denoted by the symbol $\hook$,
so that
\[
  (\alpha_1^{a_1} \dotsm \alpha_n^{a_n}) \hook (x_1^{d_1} \dotsm x_n^{d_n}) = \prod_{i=1}^n \frac{d_i!}{(d_i-a_i)!} x_i^{d_i-a_i}
\]
if each $d_i \geq a_i$, or $0$ otherwise.

For $F \in S$, $F^\perp \subseteq T$ is the ideal of $\Theta \in T$ such that $\Theta \hook F = 0$.
For example, $(x_1^{d_1}\dotsm x_n^{d_n})^\perp = (\alpha_1^{d_1+1},\dotsc,\alpha_n^{d_n+1})$.
If $F$ is homogeneous then $F^\perp$ is a homogeneous ideal.
For more details see for example \cite[\S1.1]{MR1735271}.

\subsection{Concise forms}\label{sect: concise forms}

In the terminology of \cite{MR2279854}, a form $F \in S^d V \cong \bbk[x_1,\dotsc,x_n]_d$ is called \emph{concise}
with respect to $V$ (or with respect to $x_1,\dotsc,x_n$) if $F$ cannot be written as a homogeneous form in fewer variables,
even after a linear change of coordinates; that is, $F$ is concise if $V' \subseteq V$ and $F \in S^d V'$ implies $V' = V$.
The following are equivalent: $F$ is concise; the projective hypersurface $V(F)$ is not a cone (i.e., has empty vertex);
the ideal $F^\perp$ has no linear elements; the $(d-1)$th derivatives of $F$ span the linear forms.
Note that the last two conditions can be checked directly by computation.

Write $\langle F \rangle$ for the span of the $(d-1)$th (degree $1$) derivatives of $F$.
We have $\langle F \rangle = ((F^\perp)_1)^\perp$, that is, $\langle F \rangle$ is perpendicular
to the space of linear forms in the ideal $F^\perp$.
Nonzero elements of $\langle F \rangle$ (or, elements of a basis of $\langle F \rangle$) are called \emph{essential variables} of $F$.
We have $F \in S^d \langle F \rangle$, see \cite[Proposition 1]{MR2279854}.

\subsection{Generic and maximal Waring rank}

The rank of a quadratic form is equal to its number of essential variables, by diagonalization.
Thus, if $d=2$, then $g=m=n$.
If $n=2$, then $g = \lfloor \frac{d+2}{2} \rfloor$ and $m=d$ by work of Sylvester and others in the 19th century,
see for example \cite[\S1.3]{MR1735271} and references therein.
We will review the $n=2$ case in the next section.

For $n,d \geq 3$ the generic rank is known by the famous Alexander-Hirschowitz Theorem:
\begin{theorem}[\cite{MR1311347}]
   Suppose $n, d \ge 3$. The generic rank $g=g_{n,d}$ with respect to the Veronese variety $\nu_d(\PP^{n-1})$ is as follows.
   \begin{enumerate}
      \item If $n=3$ and $d=4$, then $g_{3,4} =6$,
      \item if $n=4$ and $d=4$, then $g_{4,4} =10$,
      \item if $n=5$ and $d=3$, then $g_{5,3} =8$,
      \item if $n=5$ and $d=4$, then $g_{5,3} =15$,
      \item and otherwise, if $(n,d) \notin \set{(3,4), (4,4), (5,3), (5,4)}$, then 
      \[
          g_{n,d} = \left\lceil \frac{1}{n}\binom{d+n-1}{d} \right\rceil = \left\lceil \frac{(d+n-1)!}{d! \cdot n!}\right\rceil.
      \]
   \end{enumerate}
   Moreover, the last proper secant variety $\sigma_{g-1}(\nu_d(\PP^{n-1}))$ is a hypersurface if and only if
   $\binom{d+n-1}{d} \equiv 1 \pmod{n}$
   or it is an exceptional case, $(n,d) \in \set{(3,4), (4,4), (5,3), (5,4)}$.
\end{theorem}
On the other hand, the maximal rank $m=m_{n,d}$ is only known in a few initial cases:
$m_{3,3} = 5$, $m_{3,4} = 7$, $m_{3,5} = 10$, $m_{4,3} = 7$.
See \cite{MR3536055} for details and references.

We have $m_{n,d} > g_{n,d}$ when $n=2$, when $n=3$, and when $n=4$ and $d$ is odd \cite{MR3536055}.
In all other cases, it is an open question whether $m_{n,d} > g_{n,d}$.

\subsection{Binary forms}\label{sect_binary_forms}

Suppose $X = \nu_d(\PP^1) \subset \PP^d$ is the rational normal curve of degree $d$.
We identify $\PP^d$ as the space of forms of degree $d$ in two variables.
Here $X$ is the set of $d$th powers $[\ell^d]$.
In this case the apolarity method provides a full description of the loci $W_k$, as follows.
Let $\tau(X)$ be the tangential variety of $X$.
Some parts of the following statement are well-known, see for example \cite[\S1.3]{MR1735271},
and much (perhaps all) of it is known to experts,
but we include the statement here for lack of a clear reference.
\begin{proposition}
Let $X = \nu_d(\PP^1) \subset \PP^d$ be the rational normal curve of degree $d$.
The generic rank is $g = \lfloor \frac{d+2}{2} \rfloor$ and the maximal rank is $m=d$.
We have $W_m = W_d = \tau(X)$.
For $g < k < m$ we have $W_k = \tau(X) + (d-k)X$.
In particular, we have the following nested inclusions of irreducible varieties (each one of codimension $1$ in the next):
\[
  X \subset \tau(X) \subset 2X \subset \tau(X)+X \subset 3X \subset \tau(X)+2X \subset \dotsb,
\]
equivalently,
\[
  \sigma_1(X) \subset W_{d} \subset \sigma_2(X) \subset W_{d-1} \subset \sigma_3(X) \subset W_{d-2} \subset \dotsb .
\]
If $d = 2g-2$ is even, then the sequence of inclusions ends with:
\[
  \dotsb \subset \sigma_{g-2}(X) \subset W_{g+1} \subset \sigma_{g-1}(X) \subset \sigma_g(X) = \PP^d.
\]
Or, if $d = 2g-1$ is odd, then it ends with:
\[
  \dotsb \subset \sigma_{g-1}(X) \subset W_{g+1} \subset \sigma_g(X) = \PP^d.
\]
\end{proposition}

\begin{proof}
Fix $S = \bbk[x,y]$ and $V = S_1$, the space of linear forms in $S$.
We identify $S_d$ with $S^d V$ and the dual ring $T = \bbk[\alpha,\beta]$ with the symmetric algebra on $V^*$.

For a homogeneous polynomial $F \in S^d V$ let $F^{\perp}$ be the apolar ideal of $F$.
Then $F^{\perp} = (\Theta, \Psi)$ is a homogeneous complete intersection with $\deg \Theta =r \le \deg \Psi = d+2-r$,
see for example \cite[Theorem 1.44]{MR1735271}.
Note that both $\Theta$ and $\Psi$ are homogeneous polynomials in two variables, hence they are products of linear factors.
Then (see for example \cite[\S1.3]{MR1735271}, \cite{MR2754189})
$F \in \sigma_{r} (X) \setminus \sigma_{r-1} (X)$;
if $\Theta$ has all distinct roots, then $\rank(F) = r$; and
if $\Theta$ has at least one repeated root, then $\rank(F) = d+2-r$.

Note further, that if $r < d+2-r$, the polynomial $\Theta$ is unique up to rescaling.
In particular, whether it has distinct roots or not does not depend on any choices,
so the conditions for $\rank(F)=r$ or $d+2-r$ are well defined. 
(If $r=d+2-r$, then the conclusions are the same in both cases.)
Still assuming $r < d+2-r$, the uniqueness of $\Theta$ determines a well defined map:
\begin{align*}
   \pi_{r}\colon \sigma_r(X) \setminus \sigma_{r-1}(X) & \to \PP (S^{r} V^*), \\
                                                     F &\mapsto [\Theta].
\end{align*}
The map is surjective and every fiber $\pi_r^{-1}[\Theta]$ is a Zariski open subset of a linear subspace $\PP^{r-1} \subset \PP (S^d V)$, 
where $\PP^{r-1}$ is the linear span of $\nu_d(V(\Theta))$.
The locus of $\Theta$ with a double root is an irreducible divisor in $\PP (S^{r} V^*)$ and (the closure of) its preimage $W_{d+2-r}$ is also irreducible of codimension $1$ in $\sigma_r(X)$.

From this we see that $\dim \sigma_r(X) = 2r-1$,
so the generic rank $g = \lceil \frac{d+1}{2} \rceil = \lfloor \frac{d+2}{2} \rfloor$.
Furthermore, the maximal rank is $m=d$, and it appears whenever $r$ is $2$ and $\Theta$ has a double root,
so that $F$ is in the span of a double point, i.e., $F$ is in the tangential variety $\tau(X)=W_m=W_d$.

In between, for $g < k < m$ we have $\tau(X) + (d-k)X = W_m + (m-k)X \subseteq W_k$.
Both $\tau(X) + (d-k)X$ and $W_k$ are irreducible.
We have $\dim \tau(X) + (d-k)X \geq 2(d-k+1) = \dim \sigma_{d-k+2}(X)-1 = \dim W_k$
by Corollary~\ref{dimension bound for chains of joins} and the dimension computations above.
Hence $W_k = \tau(X) + (d-k)X$.

Since $X \subset \tau(X) \subset 2X$ we have $lX \subset \tau(X) + (l-1)X \subset (l+1)X$ for $1 \leq l < g$,
where the inclusions are of irreducible varieties, each of codimension $1$ in the next.
This proves the inclusions displayed in the statement.
\end{proof}

\subsection{Powers of quadratic forms}
Fix $n$, let $Q_n = x_1^2+\dotsb+x_n^2$, and consider $Q_n^k$, a form of degree $d = 2k$.
Reznick showed every form of degree $k$ in $n$ variables is a derivative of $Q_n^k$, see \cite[Theorem 3.10]{MR1347159}.
This can be used to show that $\rank Q_n^k \geq \binom{n-1+k}{n-1}$, see for example \cite[Theorem 5.3C,D]{MR1735271}.
(See \cite[Theorem 8.15(ii)]{MR1096187} for the real case.)
Sometimes equality holds, see \cite[Chapters 8, 9]{MR1096187}.
For example, Reznick uses the Leech lattice in $\mathbb{R}^{24}$ to show that
\[
  \rank((x_1^2+\dotsb+x_{24}^2)^{5}) = 98\,280 = \binom{28}{5} .
\]
Note that $g_{24,10} = 3\,856\,710$.

Reznick gives an expression \cite[(10.35)]{MR1096187}:
\[
  (x_1^2+\dotsb+x_n^2)^2 = \frac{1}{6} \sum_{i<j} (x_i \pm x_j)^4 + \frac{4-n}{3} \sum_{i=1}^n x_i^4,
\]
thus $\rank(Q_n^2) \leq n^2$,
so for sufficiently large $n$, $g_{n,4} = O(n^3) \gg \rank(Q_n^2)$.
For small $n$, Reznick shows
that $g_{n,4} \leq \rank(Q_n^2) \leq g_{n,4}+1$ for $n=3,4,5,6$.

There is a similar identity:
\begin{multline*}
  60 (x_1^2+\dotsb+x_n^2)^3 = \sum_{i<j<k} (x_i \pm x_j \pm x_k)^6 \\
    + 2(5-n) \sum_{i<j} (x_i \pm x_j)^6 + 2(n^2-9n+38) \sum x_i^6,
\end{multline*}
so $\rank(Q_n^3) \leq 4 \binom{n}{3} + 2 \binom{n}{2} + n$.
Hence $g_{n,6} = O(n^5) \gg \rank(Q_n^3)$ for sufficiently large $n$.

It would be interesting to determine $\rank(Q_n^k)$,
in particular to determine whether the rank is greater than the generic rank,
and whether it is strictly less than the maximal rank.

\subsection{Orbits of homogeneous forms}\label{sect: orbits of homogeneous forms}

When $F$ defines a smooth hypersurface in $\PP V$ of degree $\deg F \geq 3$,
the stabilizer of $F$ in $\SL(V)$ is finite, see for example \cite[(2.1)]{MR0476735}.
In particular the projective orbit $\GL(V) \cdot [F] \subset \PP(S^d V)$ has dimension $n^2-1$ where $n = \dim V$.
In this section we establish a lower bound for the dimension of $\GL(V) \cdot [F]$, assuming only that $F$ is concise.

\begin{lemma}\label{lem_general_section_is_a_cone}
   Assume $Y\subset \PP^N$ is a reduced subscheme such that every irreducible component of $Y$ has dimension at least $1$.
   Suppose a general hyperplane section $Y\cap \PP^{N-1}$ is a cone. Then $Y$ is a cone with $\dim \Vertex(Y)\ge 1$.
\end{lemma}

\begin{proof}
   If $Y=\PP^N$, then there is nothing to prove, so suppose $\dim Y < N$.
  First assume that $Y$ is irreducible.
  Replacing $\PP^N$ with a subspace if necessary, we may assume that $Y$ is nondegenerate.
   Consider the vertex-incidence subvariety $Z \subset Y \times (\PP^N)^*$, defined by:
   \[
     Z = \set{(y, H) \in Y \times (\PP^N)^* \mid y \in \Vertex(Y \cap H)}
   \]
   with its natural projections $\pr_1 : Z \to Y$ and $\pr_2 : Z \to (\PP^N)^*$.
   By our assumptions $\pr_2$ is dominant,
   so $\dim Z \ge N > \dim Y$.
   Let $W$ be the image of $\pr_1$.
   In particular, by dimension count, for a general point $w\in W$, there is a positive dimensional fiber $\pr_1^{-1}(w) \subset Z$.
   Let $Z_w = \pr_2(\pr_1^{-1}(w))$, so that $Z_w$
      is a positive dimensional family of hyperplanes $H$
      such that $w$ is a vertex of the cone $Y\cap H$.
   
   Since $Y$ is irreducible and nondegenerate,
      a general point $y$ in $Y$ is contained in some $Y\cap H$ with $H \in Z_w$.
   Then the line through $y$ and $w$ is contained in $Y \cap H$.
   Hence $w+Y = Y$, so $w \in \Vertex(Y)$.
   Therefore $Y$ is a cone, $W \subset \Vertex(Y)$, and $\dim \Vertex(Y) \ge \dim W$.
   But $\dim W>0$, because every general hyperplane contains a point of $W$.
   
   Now if $Y$ is reducible then by the above, each irreducible component is a cone.
   A vertex $w$ of a general hyperplane section $Y \cap H$ is a vertex of each component
   and the result follows.
\end{proof}

\begin{lemma}\label{lem_general_substitution_of_variables}
   Assume $V$ is a vector space and $n=\dim V \ge 3$, $d\ge 2$.
   Suppose $H_d \subset \PP V^* \simeq \PP^{n-1}$ is a (not necessarily reduced) 
      hypersurface of degree $d$, which is not a cone. 
   Then a general hyperplane section $H_d \cap \PP^{n-2}$ is not a cone.

   Equivalently, suppose $F\in \PP(S^d V)$ is a concise polynomial (in $n$ variables $\fromto{x_1}{x_n}$, of degree $d \ge 2$).
   Pick a general linear substitution of variables, say $x_n = a_1 x_1 + \dotsb + a_{n-1} x_{n-1}$.
   Let $F'(\fromto{x_1}{x_{n-1}})=F(\fromto{x_1}{x_{n-1}}, a_1 x_1 + \dotsb + a_{n-1} x_{n-1})$.
   Then $F'$ essentially depends on $n-1$ variables and no fewer.
\end{lemma}
\begin{proof}
   This follows from Lemma~\ref{lem_general_section_is_a_cone}, 
      since a hypersurface is a cone if and only if its reduced subscheme is a cone.
\end{proof}

Now we turn our attention to $\GL(V)$-orbits.

\begin{lemma}\label{lem_dimension_of_nonconcise_orbit}
   Let $F\in \PP(S^d V)$ be a polynomial in $n= \dim V$ variables, which essentially depends on $k$ variables with $0 < k < n$ (i.e., $F$ is non-trivial and non-concise).
   Then $F$ determines uniquely the linear subspace $V' \subset V$ of dimension $k$ such that $F \in \PP(S^d V') \subset \PP(S^dV)$.
   In particular, 
   \[
     \dim (\GL(V) \cdot [F]) = \dim (\GL(V') \cdot [F]) + \dim \Gr(k, V).
   \]
\end{lemma}
\begin{proof}
   We have $V' = \langle F \rangle = ((F^{\perp})_1)^{\perp}$ or $S^{d-1}V^* \hook F$,
   as in \S\ref{sect: concise forms}.
   The fibration 
   \begin{align*}
      \GL(V) \cdot [F]  &\to \Gr(k, V)\\
      [F'] = [g \cdot F]  &\mapsto \langle F' \rangle  = g \cdot \langle F \rangle
   \end{align*}
   is onto, 
     and each fiber is isomorphic to $\GL(V') \cdot [F]$, proving the dimension claim.
\end{proof}

\begin{proposition}\label{prop_dimension_of_orbit_of_a_concise_form}
   Suppose $F\in \PP(S^d V)$ is a concise polynomial in $n = \dim V \ge 3$ variables. 
   Let $Z = \overline{\GL(V) \cdot [F]} \subset \PP (S^d V)$.
   Then either $\dim Z \ge \binom{n+1}{2}$, or 
   $\dim Z =  \binom{n+1}{2} -1$, $d = 2k$ is even, and $F = Q^k$ for a concise quadratic polynomial $Q$.
\end{proposition}

\begin{proof}
   Let $F_n=F$, $V_n=V$,
   and define inductively $F_i$ to be a polynomial in $i$ variables (a basis of $V_i$) obtained from $F_{i+1}$ by a general substitution of one variable, 
      as in Lemma~\ref{lem_general_substitution_of_variables}.
   Thus $F_i$ is a polynomial essentially dependent on $i$ variables.
   
   The closure of the orbit $ \overline{\GL(V_n) \cdot [F_n]}$ contains $\End(V_n) \cdot [F_n]$.
   In particular, the closure contains a general substitution of variables, i.e.~it contains $[F_i]$ for all $i \le n$.
   But $\GL(V_n) \cdot [F_n]$ does not contain $[F_i]$ for $i <n$.
   Thus 
   \begin{multline*}
     \dim \overline{\GL(V_n) \cdot [F_n]} \ge \dim \overline{\GL(V_n) \cdot [F_{n-1}]} +1 \\
      = \dim (\GL(V_{n-1}) \cdot [F_{n-1}]) + (n-1) +1,
   \end{multline*}
   by Lemma \ref{lem_dimension_of_nonconcise_orbit}.
   Inductively,
   \begin{multline*}
     \dim \overline{\GL(V_n) \cdot [F_n]} \ge n + (n-1) + \dotsb +5 +4 + \dim (\GL(V_{3}) \cdot [F_{3}]) \\
     = \binom{n+1}{2} -6 + \dim (\GL(V_{3}) \cdot [F_{3}]).
   \end{multline*}
   Note that if $F_3 = Q_3^k$ for some quadric in three variables $Q_3$, then by the generality of our choices of linear substitutions (or equivalently, of hyperplane sections of the loci $(F_i=0)$), 
      we also must have $F=F_n=Q^k$.
   Thus it only remains to show the claim of the proposition for $n=3$.
   
   Denote by $(F_3)_{\text{red}}\in \PP (S^{r} V_3)$ the homogeneous equation of the reduced algebraic set $(F_3=0) \subset \PP (V_3^*) \simeq \PP^2$.
   Observe that $(F_3)_{\text{red}}$ essentially depends on $3$ variables, just as $F_3$ does.
   In particular, $r = \deg (F_3)_{\text{red}} \ge 2$, and if $r=2$, then $(F_3)_{\text{red}}$ is a nondegenerate (irreducible) quadric $Q_3$, and hence $F = Q^k$ and the claim of the proposition is proved.
   From now on, we assume $r\ge 3$.

   Consider the general line section of the plane curve $(F_3=0)$. 
   The degree $r=\deg (F_3)_{\text{red}}$ is the number of distinct points of support of this line section.
   By Lemma~\ref{lem_dimension_of_nonconcise_orbit} again:
   \[
     \dim \overline{\GL(V_3) \cdot [F_3]} \ge \overline{\GL(V_3) \cdot [F_{2}]} +1 = \dim (\GL(V_{2}) \cdot [F_{2}]) + 3
   \]  
   and also $\dim (\GL(V_{2})\cdot [F_2]) = 3$ since $F_2$ has $r$ (at least $3$) distinct roots.
   Therefore $\dim \overline{\GL(V_3) \cdot [F_3]} \ge 6$ and  $\dim \overline{\GL(V_n) \cdot [F_n]} \ge \binom{n+1}{2}$ as claimed.
\end{proof}

\subsection{Conciseness of forms of high rank}

We will use the following lemma in the next section.
\begin{lemma}\label{lem_form_of_max_rank_is_concise}
   Suppose $F \in \PP(S^d V)$ is a form of maximal rank and $d\ge 2$. Then $F$ is concise.
   In particular, a general point in each component of $W_m$ is a concise form.
\end{lemma}
\begin{proof}
   Suppose on the contrary, that there exists a choice of variables $\fromto{x_1}{x_n}$ in $V$ (where $\dim V = n$), such that $F = F(\fromto{x_1}{x_{n-1}})$ and $\rank(F) = m = m_{n,d}$.
   Then by  \cite[Proposition 3.1]{MR3320211} $\rank(F + x_n^d)  = \rank(F) + 1 > m$, a contradiction.
\end{proof}

We can generalize the above lemma
to show that all forms of greater than generic rank are necessarily concise under certain conditions.
For this we use the following simplified bound for the maximal rank.
\begin{lemma}
   $m_{n,d} \le \left\lceil\frac{2 (d+n-1)!}{n! \cdot d!}\right\rceil$.
\end{lemma}
\begin{proof}
   In the exceptional cases $(n,d) \in \set{(3,4), (4,4), (5,3), (5,4)}$, use the hypersurface version of Theorem~\ref{thm_dimension_bound_for_homogeneous}
   and check that $2 g_{n,d} - 2$ is less than or equal to the right hand side.
   In the nonexceptional cases, use the general version of Theorem~\ref{thm_dimension_bound_for_homogeneous}
   and check that $2 g_{n,d} - 1$ is less than or equal to the right hand side.
\end{proof}

\begin{proposition}
Let $n = \dim V \geq 2$.
Suppose that $d$ satisfies the following:
if $n=2$ or $n=3$, then $d \geq 2$; otherwise, $d \geq n+1$.
Let $F$ be a form of degree $d$ in $n$ variables with greater than generic rank.
Then $F$ is concise.
\end{proposition}
\begin{proof}
If $n=2$ and $F$ is not concise then $\rank(F) = 1$.
If $d=2$ then there are no forms with greater than the generic rank, so there is nothing to prove.
When $n=d=3$ the unique (up to coordinate change) form of greater than generic rank is $F = x^2 y + y^2 z$,
which is concise (see for example \cite[\S8]{MR2628829}).

Now assume $d \geq n+1$.
Since $2 \leq \frac{d+n-1}{n}$ we have
\[
  \frac{2(d+n-2)!}{(n-1)!d!} \leq \frac{(d+n-1)!}{n!d!},
\]
hence $m_{n-1,d} \leq g_{n,d}$.
It follows that if $F$ is not concise, then $\rank(F) \leq m_{n-1,d} \leq g_{n,d}$.
\end{proof}

\subsection{Dimensions of maximal rank loci for Veronese varieties}\label{sect: proof of dimension bound for veronese}

We can now prove Theorem~\ref{thm_dimension_bound_for_Veronese}.
\begin{proof}[Proof of Theorem~\ref{thm_dimension_bound_for_Veronese}]
    Pick an irreducible component $W\subset W_m$ and 
      let $F \in Y$ be a general form from that component.
    Then $F$ is concise by Lemma~\ref{lem_form_of_max_rank_is_concise}.
    The closure $\overline{\GL(V) \cdot F}$ of the orbit of $F$ is contained in $W$.
    In particular, by Proposition~\ref{prop_dimension_of_orbit_of_a_concise_form},
    \[
       \dim W \ge \dim \overline{\GL(V) \cdot F} \ge \binom{n+1}{2} -1,
    \]
    and if $\dim W = \binom{n+1}{2} -1$, then $W = \overline{\GL(V) \cdot Q^k}$.
\end{proof}

\begin{example}
For $n=d=3$ we have $g=4$ and $m=5$.
The rank locus $W_5$ is the closure of the orbit of the form $x^2y+y^2z$,
the equation of a smooth plane conic plus a tangent line.
This orbit has dimension $6$, so $\dim W_m = 6 = \binom{n+1}{2}$.
\end{example}


\section{Tensors of format \texorpdfstring{$2\times 4\times 4$}{2x4x4}}
\label{sect: tensors}

When $X = \Seg(\bbP^{n_1-1} \times \dotsb \times \bbP^{n_k-1}) \subset \bbP^N$ is a Segre variety,
$N = n_1\dotsm n_k-1$,
then $\bbP^N$ is the projective space of \emph{tensors of format $n_1 \times \dotsb \times n_k$}
and $X$ corresponds to the simple tensors.
Rank with respect to $X$ is the usual tensor rank.

Tensors of format $2 \times b \times c$
may be regarded as pencils of matrices, which admit a normal form due to Kronecker.
Using this normal form we characterize the loci of $2 \times 4 \times 4$ tensors of higher than generic rank.
Tensors of format $2 \times 4 \times 4$ have generic rank $4$ and maximal rank $6$;
we show that $W_6 + X = W_5$ and $W_5 + X = W_4$, where $W_4 = \PP^{31}$ is the space of all $2 \times 4 \times 4$ tensors.

\subsection{Concise tensors}

A tensor $T \in V_1 \otimes \dotsb \otimes V_k$ is \emph{concise}
if $T \in V'_1 \otimes \dotsb \otimes V'_k$ with $V'_i \subseteq V_i$ for each $i$ implies $V'_i = V_i$ for each $i$.

Fix $T \in V_1 \otimes \dotsb \otimes V_k$ and for each $i$ let $V'_i \subseteq V_i$ be the image of the induced map
$V_1^* \otimes \dotsb \otimes \widehat{V_i^*} \otimes \dotsb \otimes V_k^* \to V_i$.
It is easy to see that $\rank(T) \geq \dim V'_i$ for each $i$, and also that
$T \in V'_1 \otimes \dotsb \otimes V'_k$.
In particular, if $\rank(T) < \max\{\dim V_1,\dotsc,\dim V_k\}$, then $T$ is not concise.

Non-conciseness is a closed condition
(because it is defined by vanishing of minors of certain matrices; see for example \cite[\S3.4.1]{MR2865915}).
Hence the locus of non-concise tensors
contains the secant variety $\sigma_r(X)$ for each $r < \max\{\dim V_1,\dotsc,\dim V_k\}$,
where $X$ is the Segre variety.

\begin{remark}
Observe that non-square matrices are always non-concise,
and more generally if $n_i > \prod_{j \neq i} n_j$ for some $i$,
then every tensor of format $n_1 \times \dotsm \times n_k$ is non-concise.
\end{remark}

\subsection{Normal form}

Let $\{s,t\}$ be a basis for $\bbk^2$ and let $T \in \bbk^2 \otimes \bbk^b \otimes \bbk^c$ be a tensor.
Identifying $\bbk^b \otimes \bbk^c$ with the space of $b \times c$ matrices, we can write
$T = s \otimes M_1 + t \otimes M_2$ for some $b \times c$ matrices $M_1, M_2$.
The tensor $T$ corresponds to the pencil spanned by $M_1$ and $M_2$ in $\PP(\bbk^b \otimes \bbk^c)$;
changes of basis in the pencil correspond to changes of basis in $\bbk^2$.
It is convenient to write $T$ as the $b \times c$ matrix $sM_1+tM_2$ whose entries are homogeneous linear forms in $s$ and $t$.

There is a normal form due to Kronecker for tensors $T \in \bbk^2 \otimes \bbk^b \otimes \bbk^c$, 
i.e.~a representative of the $\GL(\bbk^2) \times \GL(\bbk^{b}) \times \GL(\bbk^{c})$-orbit of $T$,
or in other words, a convenient choice of basis that makes $T$ particularly ``simple''.
Further, the results of Grigoriev, Ja'Ja'™ and Teichert calculate the rank of each tensor in normal form, see \cite[Section~5]{MR2996361}.
For simplicity of some calculations, we restrict our considerations to the case of even square matrices.
Later we restrict further to the case of $4 \times 4$ matrices.

Kronecker's normal form is as follows.
Suppose $V_{n} = \bbk^2 \otimes \bbk^{2n} \otimes \bbk^{2n}$
and $X =\Seg(\PP^1 \times \PP^{2n-1} \times \PP^{2n-1}) \subset \PP V_n$ is a Segre variety.
We encode the tensors in $V_n$ as $2n \times 2n$ matrices with entries linear forms in two variables $s$ and $t$.
For a positive integer $\epsilon$ let $L_{\epsilon}$ denote the $\epsilon \times (\epsilon +1)$ matrix
\[
  L_{\epsilon} = 
  \begin{psmallmatrix}
      s & t & 0 & \cdots & 0 & 0\\
      0 & s & t & \cdots & 0 & 0\\
      0 & 0 & s & \cdots & 0 & 0\\
      \vdots & \vdots & \vdots & \ddots & \vdots & \vdots\\
      0 & 0 & 0 & \cdots & t & 0 \\
      0 & 0 & 0 & \cdots & s & t
  \end{psmallmatrix} .
\]
Let $F$ be an $f \times f$ matrix with coefficients in $\bbk$ in Jordan normal form.
For $\lambda \in \bbk$,
   denote by $d_{\lambda}(F)$ the number of Jordan blocks of size at least $2$ with the eigenvalue $\lambda$, 
   and by $m(F)$ the maximum among $d_{\lambda}(F)$.

Given a sequence of matrices $\fromto{M_1}{M_k}$ depending on variables $s$ and $t$, denote by $M_1 \oplus \dotsb \oplus M_k$ the block matrix
   \[
      M_1 \oplus \dotsb \oplus M_k = 
      \begin{pmatrix}
      M_1 & 0 & \cdots & 0\\
      0 & M_2 & \cdots & 0\\
      \vdots&\vdots& \ddots &\vdots\\
      0 &  0  & \cdots & M_k
  \end{pmatrix} .
   \]
In the above notation we allow $M_i$ to be a zero matrix $Z_{p\times q}$ of size $p \times q$, where $p,q \ge 0$ are nonnegative integers.
Thus, for example, if $M_2$ is a $0 \times 5$ matrix, then $M_1 \oplus M_2$ is the matrix $M_1$ with five columns of zeroes added. 
\begin{theorem}[{\cite[Proposition 5.1 and Theorem 5.3]{MR2996361}}]\label{thm_kronecker_normal_form_and_ranks}
   For any tensor $T\in V_{n} = \bbk^2 \otimes \bbk^{2n} \otimes \bbk^{2n}$ there exists a choice of basis
   of $\bbk^2$, $\bbk^{2n}$, and $\bbk^{2n}$
      such that $T$ is represented by a matrix 
      \[
         T = L_{\epsilon_1} \oplus L_{\epsilon_2} \oplus  \dotsb \oplus L_{\epsilon_k} \oplus
             L_{\eta_1}^\transpose \oplus L_{\eta_2}^\transpose \oplus  \dotsb \oplus L_{\eta_l}^\transpose \oplus
             (s\Id_{f}+ t F) \oplus Z_{p\times q},
      \]
      where
       $k$, $l$, $f$, $p$, and $q$ are nonnegative integers (possibly zero);
       each $\epsilon_i$ and $\eta_j$ is a positive integer;
       $\Id_f$ is the $f \times f$ identity matrix over $\bbk$;
       $F$ is a $f \times f$ matrix in its Jordan normal form; and
       $Z_{p\times q}$ is the $p\times q$ zero matrix.
       
     Moreover, the rank of $T$ is equal to the sum of the ranks of the blocks in this normal form,
     where
       $\rank(L_{\epsilon_i}) = \epsilon_i +1$, $\rank(L_{\eta_j}^\transpose) = \eta_j +1$,
       $\rank(s\Id_{f}+ t F) = f + m(F)$, and 
       $\rank(Z_{p\times q}) =0$.
   That is:
   \[
      \rank(T) = \sum_{i} \epsilon_i +\sum_{j} \eta_j + k + l + f + m(F).
   \]
\end{theorem}
See also references discussed in \cite[Remark 5.4]{MR2996361}.

It is straightforward to see that one can always further change the coordinates so that one of the eigenvalues of $F$ is $0$.
%
%

We stress that if $T = M_1 \oplus M_2$, then the rank of $T$ is not necessarily equal  to $\rank(M_1) + \rank(M_2)$.
For example, let
                 $M_1 = \begin{psmallmatrix}
                           s & t \\
                           0 & s
                        \end{psmallmatrix}$ 
and 
                 $M_2 = \begin{psmallmatrix}
                           s +t& t \\
                           0 & s+t
                        \end{psmallmatrix}$.
Then, by Theorem~\ref{thm_kronecker_normal_form_and_ranks} the rank of $T = M_1 \oplus M_2$ is $5$, while $\rank(M_1) = \rank(M_2)=3$.

\subsection{Generic and maximal rank}

For tensors in $V_n=\bbk^2 \otimes \bbk^{2n} \otimes \bbk^{2n}$,
the generic rank is $g=2n$, and general tensors have the normal form $(s\Id_{2n} +t F)$,
where $F$ is a diagonal matrix with distinct (generic) eigenvalues.
This is because a general pencil contains an invertible matrix,
and the blocks $L_{\epsilon_i}$ or $L_{\eta_j}^\transpose$ have no invertible matrices.

Furthermore, the maximal rank is $m =3n$,
and any tensor $T$ of maximal rank is of the form $(s\Id_{2n}+ t F)$,
where $F$ has a unique eigenvalue (which we can assume to be $0$) 
and $n$ Jordan blocks of size $2\times 2$.
That is, after reordering of rows and columns, we can write $T$ as 
           $\begin{pmatrix}
               s \Id_{n} & t \Id_n\\
               0         & s \Id_n
            \end{pmatrix}
           $.

Thus $W_m = W_{3n} = \overline{ G \cdot [T]}$, where $G = \GL(\bbk^2) \times \GL(\bbk^{2n}) \times \GL(\bbk^{2n})$
is the automorphism group of $X \subset \PP V_n$, and $T =
            \begin{pmatrix}
               s \Id_{n} & t \Id_n\\
               0         & s \Id_n
            \end{pmatrix}
           $.
In particular, $W_m$ is irreducible.

\begin{remark}\label{remark: wm not contained in secant}
Note that $T =
            \begin{pmatrix}
               s \Id_{n} & t \Id_n\\
               0         & s \Id_n
            \end{pmatrix}
           $
is concise, so 
\[T \notin \sigma_{2n-1}(X) = \sigma_{g-1}(X).\]
Hence $W_m \not\subset \sigma_{g-1}(X)$.
\end{remark}

We compute the dimension of $W_m = W_{3n}$.
The main technique is to reduce to a system of linear equations via the Lie algebra stabilizer.
We illustrate this in some detail in this case, as we will use the same method (with fewer details given)
to compute dimensions of other orbits of $2 \times 4 \times 4$ tensors in the next section.
\begin{proposition}\label{prop: dim of tensor wm}
For $X=\PP^1\times \PP^{2n-1} \times \PP^{2n-1}\subset \PP^{8n^2-1}$ 
the dimension of  $W_m = W_{3n}$ is $6n^2$.
\end{proposition}
\begin{proof}
Let $\rho$ denote the action of $G$ on $V_n = \bbk^2 \otimes \bbk^{2n} \otimes \bbk^{2n}$.

The dimension of the orbit $G \cdot T$ is equal to the codimension in $G$ of the stabilizer subgroup of $T$ \cite[\S3.7]{MR2265844}.
We compute the dimension of the stabilizer subgroup by finding the dimension of its tangent space
at the identity $e \in G$.
Recall that in the representation $d \rho$ of the Lie algebra $T_e(G) \cong \End(\bbk^2) \times \End(\bbk^{2n}) \times \End(\bbk^{2n})$ on $V_n$,
a tangent vector $(g_1,g_2,g_3)$ acts on $(sM_1+tM_2) \in V_2$ by
\begin{multline*}
  d\rho(g_1,g_2,g_3) . (s M_1 + t M_2) = \Big( (as+ct)M_1 + (bs+dt)M_2 \Big) \\
    + \Big( s (g_2 M_1) + t (g_2 M_2) \Big) - \Big( s (M_1 g_3) + t (M_2 g_3) \Big),
\end{multline*}
where $g_1 = \begin{pmatrix} a & b \\ c & d \end{pmatrix}$ \cite[(6.1.1)]{MR2265844}.
Recall also that a tangent vector $(g_1,g_2,g_3) \in T_e(G)$ lies in the tangent space to the stabilizer of $T$ at $e$ if and only if
the derivative $d \rho (g_1,g_2,g_3)$ annihilates $T$ \cite[\S3.5, Theorem 2]{MR2265844}.

Write in block form $T = \begin{pmatrix} s I_n & t I_n \\ 0 & s I_n \end{pmatrix}$,
so $M_1 = I_{2n}$ and $M_2 = \begin{pmatrix} 0 & I_n \\ 0 & 0 \end{pmatrix}$.
Write $g_2 = \begin{pmatrix} A_{11} & A_{12} \\ A_{21} & A_{22} \end{pmatrix}$,
$g_3 = \begin{pmatrix} B_{11} & B_{12} \\ B_{21} & B_{22} \end{pmatrix}$,
where the $A_{ij}$ and $B_{ij}$ are $n \times n$ matrices.
Then $(g_1,g_2,g_3)$ is in the tangent space to the stabilizer of $T$ if and only if
\begin{multline*}
  \Big( (as+ct)M_1 + (bs+dt)M_2 \Big) + \Big( s (g_2 M_1) + t (g_2 M_2) \Big) \\
    - \Big( s (M_1 g_3) + t (M_2 g_3) \Big) = 0 .
\end{multline*}
The left hand side is
\begin{multline*}
  \begin{pmatrix}
    a I_n + A_{11} - B_{11} & b I_n + A_{12} - B_{12} \\
            A_{21} - B_{21} & a I_n + A_{22} - B_{22}
  \end{pmatrix} s
    \\
    +
  \begin{pmatrix}
    c I_n - B_{21} & d I_n + A_{11} - B_{22} \\
    0 & c I_n + A_{21} 
  \end{pmatrix} t.
\end{multline*}
This must vanish identically, which yields the equations
\begin{align*}
  B_{11} &= a I_n + A_{11},  &  A_{21} &= 0, \\
  B_{12} &= b I_n + A_{12},  &  B_{21} &= 0, \\
  B_{22} &= d I_n + A_{11},  &  c &= 0, \\
  A_{22} &= (d-a) I_n + A_{11} . & &
\end{align*}
Note that $A_{11},A_{12},a,b,d$ are free, so the stabilizer has dimension $2n^2+3$.
Since $\dim G = 8n^2+4$, the affine orbit $G \cdot T \subset V_n$ has dimension $6n^2+1$.
The projective orbit $G \cdot [T] \subset \PP V_n$ has dimension one less,
since $G$ contains subgroups isomorphic to $\mathbb{G}_m = \bbk^*$ that act on $V_n$ as rescaling.
So $\dim W_m = \dim G \cdot [T] = 6n^2$, as claimed.
\end{proof}

\begin{remark}
This shows that for $X$ as above, some of the intermediate joins $W_{3n} + kX$ 
for $k\in \setfromto{1}{n-1}$ must be highly defective.
Indeed, the expected dimension of $W_{3n}+ \lceil \frac{n}{2}\rceil X$ is already the dimension of the ambient $\PP^{8n^2-1}$,
while we know that even $W_{3n} + (n-1)X$ does \emph{not} fill $\PP^{8n^2-1}$.
\end{remark}

\subsection{Orbits of \texorpdfstring{$2 \times 4 \times 4$}{2x4x4} tensors}

We now specialise to the case $n=2$, i.e., tensors in $V_2 = \bbk^2 \otimes \bbk^4 \otimes \bbk^4$.

Let $G = \Gl_2 \times \Gl_4 \times \Gl_4$ and consider the natural action of $G$ on $\PP(V_2)$.
Note that $\dim G = 36$ and $\dim \PP(V_2) =31$.
\begin{lemma}\label{lem_orbits_in_V_2}
   The orbit structure of the action of $G$ on $\PP(V_2)$ is as follows.
   \begin{enumerate}
    \item There is no open orbit.
    \item The only orbits of codimension $1$ are the orbits of (classes of) tensors (in their Kronecker normal forms):
    \begin{align*}
    T_4(\lambda_1, \lambda_2, \lambda_3, \lambda_4) &=
    \begin{psmallmatrix}
        s + \lambda_1 t & 0 & 0 & 0 \\
        0 & s + \lambda_2 t & 0 & 0 \\
        0 & 0 & s + \lambda_3 t & 0 \\
        0 & 0 & 0 & s + \lambda_4 t \\
    \end{psmallmatrix},
    \quad 
    \text{ or}\\
    T_5 (\lambda_1, \lambda_2, \lambda_3) &= 
    \begin{psmallmatrix}
        s + \lambda_1 t & t & 0 & 0 \\
        0 & s + \lambda_1 t & 0 & 0 \\
        0 & 0 & s + \lambda_2 t & 0 \\
        0 & 0 & 0 & s + \lambda_3 t \\
    \end{psmallmatrix}
    \end{align*}
   for pairwise distinct eigenvalues $\lambda_i$.
   Two tensors of the form $T_4(\lambda_1, \lambda_2, \lambda_3, \lambda_4)$ are in the same orbit if and only if 
      the cross-ratios of their eigenvalues
      $\frac{\lambda_1-\lambda_2}{\lambda_1-\lambda_3}\cdot \frac{\lambda_4-\lambda_3}{\lambda_4-\lambda_2}$ are equal 
      (after possibly permuting the order of $\lambda_i$).
   Any two tensors of the form $T_5(\lambda_1, \lambda_2, \lambda_3)$ are in the same orbit.
   \item There are finitely many orbits of codimension at least $2$.
   \end{enumerate}
\end{lemma}
\begin{proof}
   The set of projective classes of nonconcise tensors 
   (i.e.~those contained in some $\PP(\bbk^1 \otimes \bbk^4 \otimes \bbk^4)$ or $\PP(\bbk^2 \otimes \bbk^3 \otimes \bbk^4)$ 
      or $\PP(\bbk^2 \otimes \bbk^4 \otimes \bbk^3)$) is $G$-invariant, of dimension $27$ (hence codimension $4$), and has only finitely many orbits
      \cite[Section~6]{MR2996361}.
   Thus it is enough to prove the lemma for concise tensors.

   To see when tensors of the form $T_4=T_4(\lambda_1, \lambda_2, \lambda_3, \lambda_4)$ are in the same orbit,
      note $\det(T_4) = (s+\lambda_1 t) \dotsm(s + \lambda_4 t)$ determines four points $[-\lambda_i,1]$ on $\PP^1$ parametrised by $s,t$.
   In particular, tensors with different cross-ratios of eigenvalues (up to permutation) cannot be in the same orbit.
   On the other hand, if there are two sets of eigenvalues with the same cross-ratio, 
      then we can change the coordinates $(s,t)$ on $\kk^2$, and then also rescale columns to get from one tensor to the other.
   Let $G^0_{[T_4]} \subset G$ be the identity component of the stabilizer of $[T_4] \in \PP(V_2)$.
   Suppose $(g_1,g_2, g_3) \in G^0_{[T_4]}$, where $g_1 \in \GL_2$, $g_2 \in \GL_4$ and $g_3 \in \GL_4$.
   The action of $g_1$ on $\PP^1$ must preserve the four points (zeroes of determinant).
   Thus $g_1=\mu_1 \Id_2$ is a rescaling of the identity.
   Restricting to the $s$ coordinate, we see that the product  $g_2 g_3 = \mu_2 \Id_4$ is also a rescaling of the identity, 
      that is $g_3 = \mu_2 g_2^{-1}$.
   Hence restricting to the $t$ coordinate, $g_2$ commutes
     with a diagonal matrix with pairwise distinct entries.
   Then it is straightforward to see that $g_2$ is an invertible diagonal matrix, and any invertible diagonal matrix
   can occur as $g_2$. Thus $\dim G^0_{[T_4]} =6$ 
      and the dimension of the orbit of $[T_4]$ is $30 = 36-6$, as claimed.

   In particular, since a general tensor is of the form $T_4$, it lies in an orbit of codimension $1$.
   So there is no open orbit.
   
   To see that $T_5(\lambda_1, \lambda_2, \lambda_3)$ is always in the same orbit, 
     we use a linear transformation $\phi\colon \PP^1 \to \PP^1$,
      which takes the triple of points $([-\lambda_1,1],[-\lambda_2,1],[-\lambda_3,1])$ to $([0,1],[-1,1],[1,1])$. 
   Let $[1, \nu_1]$ be the image of $[1,0]$.
   Lifting $\phi$ to $\widehat{\phi}\colon\kk^2 \to \kk^2$ we obtain that 
   \[
     \widehat{\phi}(T_5) = 
     \begin{psmallmatrix}
        \nu_2 s  & \nu_3( -\nu_1 s  +t) & 0 & 0 \\
        0 & \nu_2s & 0 & 0 \\
        0 & 0 & \nu_4 (s + t) & 0 \\
        0 & 0 & 0 & \nu_5(s -t)
    \end{psmallmatrix}
   \]
   for some nonzero constants $\nu_2,\dotsc,\nu_5$.
   Then using column and row rescalings we can modify the matrix to   
   $\begin{psmallmatrix}
        s  &  -\nu_1 s  +t & 0 & 0 \\
        0 & s & 0 & 0 \\
        0 & 0 &  s + t & 0 \\
        0 & 0 & 0 & s -t
   \end{psmallmatrix}$.
   Finally, we add a multiple of the first column to the second column to obtain 
   $\begin{psmallmatrix}
        s & t & 0 & 0 \\
        0 & s & 0 & 0 \\
        0 & 0 & s + t & 0 \\
        0 & 0 & 0 & s -t
   \end{psmallmatrix}$.
   Thus any $T_5(\lambda_1, \lambda_2, \lambda_3)$ is in the same $G$-orbit as $T_5(0, 1, -1)$.
   As in the proof of Proposition \ref{prop: dim of tensor wm},
   we can check that the dimension of the Lie algebra stabilizer of $[T_5(0,1,-1)] \in \PP (V_2)$ is $6$,
      hence its orbit is of codimension $1$.

   It remains to check that
   there are finitely many other concise orbits and that
   all these other orbits have codimension at least $2$, i.e.~dimension at most $29$.

   For the first part we use the normal form described in Theorem~\ref{thm_kronecker_normal_form_and_ranks}
   and rescaling to fix the eigenvalues.
   It is straightforward to see that there are $14$ concise orbits other than the $T_4$ and $T_5$ cases.
   The second part is an explicit computer calculation of the dimension of the Lie algebra stabilizer for each of the cases above,
   as in the proof of Proposition \ref{prop: dim of tensor wm}.
   Representatives for the $14$ orbits are listed, along with the dimensions of the orbits and their ranks,
   in Table~\ref{table_concise_orbits_in_V_2}.
\end{proof}

\begin{table}[htb]
   \[
     \begin{array}{lcc @{\hspace{2cm}} lcc}
     \toprule
      \text{orbit}  &  \dim & \text{rank}  &    \text{orbit}  &  \dim &   \text{rank} \\
     \midrule  \addlinespace
      \begin{psmallmatrix}
         s & t & 0 & 0\\
         0 & s & 0 & 0\\
         0 & 0 & s & t\\
         0 & 0 & 0 & s
      \end{psmallmatrix}&  24 & 6 &
              \begin{psmallmatrix}
                 s & t & 0 & 0\\
                 0 & s & t & 0\\
                 0 & 0 & 0 & s\\
                 0 & 0 & 0 & t
              \end{psmallmatrix} & 26 & 5 \\ \addlinespace
      \begin{psmallmatrix}
         s & t & 0 & 0\\
         0 & s & t & 0\\
         0 & 0 & s & 0\\
         0 & 0 & 0 & s+t
      \end{psmallmatrix}&  29 & 5 &
              \begin{psmallmatrix}
                 s & t & 0 & 0\\
                 0 & 0 & s & 0\\
                 0 & 0 & t & s\\
                 0 & 0 & 0 & t
              \end{psmallmatrix} & 26 & 5 \\ \addlinespace
      \begin{psmallmatrix}
         s & t & 0 & 0\\
         0 & s & 0 & 0\\
         0 & 0 & s+t & t\\
         0 & 0 & 0 & s+t
      \end{psmallmatrix}& 29 & 5 &
              \begin{psmallmatrix}
                 s & t & 0 & 0\\
                 0 & 0 & s & 0\\
                 0 & 0 & t & 0\\
                 0 & 0 & 0 & s
              \end{psmallmatrix} &  25 & 5 \\ \addlinespace
      \begin{psmallmatrix}
         s & t & 0 & 0\\
         0 & s & t & 0\\
         0 & 0 & s & t\\
         0 & 0 & 0 & s
      \end{psmallmatrix}& 28 & 5 &
              \begin{psmallmatrix}
                 s & t & 0 & 0\\
                 0 & s & 0 & 0\\
                 0 & 0 & s & 0\\
                 0 & 0 & 0 & s
              \end{psmallmatrix}& 22 & 5 \\ \addlinespace
      \begin{psmallmatrix}
         s & t & 0 & 0\\
         0 & s & 0 & 0\\
         0 & 0 & s+t & 0\\
         0 & 0 & 0 & s+t
      \end{psmallmatrix}&  27 & 5 &
              \begin{psmallmatrix}
                 s & 0 & 0 & 0\\
                 0 & s & 0 & 0\\
                 0 & 0 & s+t & 0\\
                 0 & 0 & 0 & s-t
              \end{psmallmatrix} & 28 & 4 \\ \addlinespace
      \begin{psmallmatrix}
         s & t & 0 & 0\\
         0 & s & 0 & 0\\
         0 & 0 & s & 0\\
         0 & 0 & 0 & s+t
      \end{psmallmatrix}&  27 & 5 &
              \begin{psmallmatrix}
                 s & 0 & 0 & 0\\
                 0 & s & 0 & 0\\
                 0 & 0 & s+t & 0\\
                 0 & 0 & 0 & s+t
              \end{psmallmatrix} &  25 & 4 \\ \addlinespace
      \begin{psmallmatrix}
         s & t & 0 & 0\\
         0 & s & t & 0\\
         0 & 0 & s & 0\\
         0 & 0 & 0 & s
      \end{psmallmatrix}&  26 & 5 &
              \begin{psmallmatrix}
                 s & 0 & 0 & 0\\
                 0 & s & 0 & 0\\
                 0 & 0 & s & 0\\
                 0 & 0 & 0 & s+t
              \end{psmallmatrix}&  23 & 4 \\ \addlinespace
     \bottomrule
    \end{array}
   \]
   \caption{Representatives for the concise orbits in $\PP(\kk^2\otimes \kk^4\otimes \kk^4)$
            of codimension at least $2$, the dimensions (denoted $\dim$) of their orbits, and their ranks.}\label{table_concise_orbits_in_V_2}
\end{table}

Consider the determinant of $sM_1 + t M_2$ as a homogeneous polynomial of degree $4$ in the variables $s,t$, 
  whose coefficients are degree $4$ homogeneous polynomials $a_i$ in $32$ variables, the coordinates of $V_2$:
\[
   \det(sM_1 + t M_2) = a_0 s^4 + a_1 s^3t + a_2 s^2 t^2 + a_3 st^3 + a_4 t^4 
\]
(so that in particuular, $a_0 = \det M_1$,  $a_4 = \det M_2$).
Consider the discriminant of this polynomial, 
\begin{align*}
   \Discr&=
    256a_0^3a_4^3 - 192a_0^2a_1a_3a_4^2 - 128a_0^2a_2^2a_4^2 + 144a_0^2a_2a_3^2a_4 - 27a_0^2a_3^4\\
&\qquad + 144a_0a_1^2a_2a_4^2 - 6a_0a_1^2a_3^2a_4 - 80a_0a_1a_2^2a_3a_4 + 18a_0a_1a_2a_3^3\\
&\qquad + 16a_0a_2^4a_4 - 4a_0a_2^3a_3^2 - 27a_1^4a_4^2 + 18a_1^3a_2a_3a_4
    - 4a_1^3a_3^3\\
&\qquad - 4a_1^2a_2^3a_4 + a_1^2a_2^2a_3^2,
\end{align*}
which is a degree $24$ polynomial in the $32$ variables.
\begin{corollary}\label{cor_discriminant_divisor}
   A class of a tensor $T= sM_1 + t M_2$ is in the support of the effective divisor $(\Discr)$ 
      if and only if 
      $\det (sM_1 + t M_2)$ has a root of multiplicity at least two or is identically zero.
    $(\Discr)$ in $\PP^{31} = \PP (V_2)$ is $G$-invariant.
   Set theoretically, the support of $(\Discr)$ is equal to the closure of the orbit 
    $G\cdot [T_5]=
    G\cdot\begin{bsmallmatrix}
        s  & t & 0 & 0 \\
        0 & s & 0 & 0 \\
        0 & 0 & s + t & 0 \\
        0 & 0 & 0 & s - t \\
    \end{bsmallmatrix}$.
    Every class of a non-concise tensor lies in the support of $(\Discr)$,
    as does every class of a tensor of rank $5$ or $6$.
\end{corollary}

\begin{proof}
   The first characterization of the divisor is clear from the definition and properties of the discriminant, 
      while the $G$-invariance follows from this first characterization.
   Similarly, $T_5$ is in the support of $(\Discr)$, and so is the closure of its orbit.
   By Lemma~\ref{lem_orbits_in_V_2} the orbit of $T_5$ is the only orbit of codimension $1$ which is contained in the support.
   Since there are only finitely many orbits of codimension at least $2$, 
     the only irreducible $G$-invariant divisors are the closures of $30$-dimensional orbits.
   Hence the support of $(\Discr)$ is irreducible and equal to $\overline{G\cdot [T_5]}$.
   
If $T = (sM_1+tM_2)$ is non-concise, then either the normal form for $T$ involves a block of zeros,
so $\det(sM_1 + tM_2) = 0$,
or else the matrices $M_1$ and $M_2$ are linearly dependent, so the determinant has a single root of multiplicity $4$.
Hence all classes of non-concise tensors lie in the support of $(\Discr)$.
Table~\ref{table_concise_orbits_in_V_2} lists the tensors of rank $5$ or $6$ other than $T_5$.
The determinant of each tensor listed in the table is zero, or has a multiple root.
So the classes of these tensors also lie in $(\Discr)$.
\end{proof}

\subsection{High rank loci of \texorpdfstring{$2 \times 4 \times 4$}{2x4x4} tensors}\label{sect: 2x4x4 tensors}

In this case the generic rank is $g=4$ and the maximal rank is $m=6$.
\begin{proposition}\label{prop_2x4x4_tensors}
Let $V_2 = \bbk^2 \otimes \bbk^4 \otimes \bbk^4$
and $X = \Seg(\PP^1 \times \PP^3 \times \PP^3) \subset \PP V_2$.
Then
\[
    W_5=W_6+X \quad \text{and} \quad W_5+X = \PP V_2 \simeq \PP^{31}.
\]
   Moreover, $W_5$ is an irreducible divisor consisisting of those $sM_1 + tM_2$ such that $\det (s M_1+tM_2)$ 
   (considered as a homogeneous polynomial in two variables $s$ and $t$)
   is either identically $0$ or has a root of multiplicity at least $2$.
\end{proposition}

\begin{proof}
Let $T_5 = T_5(0,1,-1)$.
By Corollary~\ref{cor_discriminant_divisor}, $\overline{G \cdot [T_5]} = (\Discr)$
and every tensor of rank $5$ lies in the support of the divisor $(\Discr)$, so $W_5 \subseteq \overline{G \cdot [T_5]}$.
Conversely, the orbit $G \cdot [T_5] \subseteq W_5$.
Therefore, $W_5 = \overline{G\cdot [T_5]} = (\Discr)$ is an irreducible divisor.
Hence the equality $W_5 + X = W_4 = \PP V_2$ follows from Corollary~\ref{dimension bound for chains of joins}.

Let $T_6$ be the (unique up to a choice of coordinates) tensor of rank $6$, and let $T_1$ be a general tensor of rank $1$.
Then $T_6 +T_1$ has rank $5$ by Theorem~\ref{thm_nesting_of_join_and_loci}.
A computer calculation shows that determinant of $T_6+T_1$ is divisible by $s^2$ and has two other distinct roots not equal to $s$.
Thus $T_6 + T_1$ must be of the form $T_5$.
That is, a general element of the (irreducible) join $W_6 + X$ (where $X=\PP^1 \times \PP^3 \times \PP^3$) is a general element
  of the irreducible variety $W_5$, thus $W_6 + X= W_5$ as claimed.
\end{proof}
\begin{remark}
      The proofs above show that for $X=\PP^1\times \PP^3 \times \PP^3\subset \PP^{31}$  (so that $\dim X = 7$) 
         we have $\dim W_6 = 24$ and $\dim W_5 = \dim (W_6 + X) = 30$.
      That is, the join $W_6 + X$ is defective (it is expected to fill the ambient space excessively, but it does not).
\end{remark}

\section{Curves in quadric surfaces}\label{sect_curves}

We study rank with respect to a curve $C$ contained in a smooth quadric surface $Q \cong \PP^1 \times \PP^1$ in $\PP^3$.
By nondefectivity, the generic rank with respect to $C$ is $2$,
and by Theorem \ref{upper bound codim+1} or Theorem \ref{thm_dimension_bound_for_curve} the maximal rank is at most $3$.

If $C$ has bidegree $(2,2)$ then $C$ is an elliptic normal quartic curve.
Bernardi, Gimigliano, and Id\`a gave a description of $W_3$ in this case,
and more generally studied elliptic normal curves of degree $d+1$ in $\PP^d$, $d \geq 3$
\cite[Theorem 28]{MR2736357}.
We refine their result in the $d=3$ case
and show that the maximal rank locus $W_3$ is a curve of degree $8$ disjoint from $C$.

Piene has shown that if $C$ is a general curve of bidegree $(3,3)$ then $W_3$ is empty (the maximal rank is $2$),
see \cite[Theorem 2]{MR620126}.
We extend this to general curves of bidegree $(a,b)$, where $a \geq 4$ and $b \geq 1$.

\subsection{Elliptic quartic curve}
Let us study the locus $W_3$
with respect to an elliptic normal curve of degree $4$ in $\PP^3$.

\begin{proposition}
Let $C=Q_1\cap Q_2$ be a smooth complete intersection of two smooth quadrics in $\PP^3$.
The generic rank with respect to $C$ is $2$ and the maximal rank is $3$.
$W_3$ is a curve of degree $8$, disjoint from $C$ and
containing the vertices of the $4$ singular quadrics that contain $C$.
Every point of $W_3$ has rank $3$, except those $4$ points, which have rank $2$.
\end{proposition}

\begin{proof}
Note that $C$ has no trisecant, bitangent, or flex lines, since any such line would have to be contained in
every quadric surface that contains $C$.
The quadrics containing $C$ form a pencil with $4$ singular members, which are distinct and irreducible.
Let the vertices of those cones be $V = \{x_1,\dotsc,x_4\}$.
Each vertex $x_i$ lies off of $C$, and each $x_i$ has rank $2$.

Let $x\in\PP^3\setminus (C \cup V)$ and let $\pi:\PP^3\dasharrow\PP^2$ the projection from $x$.
Then $\pi(C)$ is an elliptic quartic curve, hence has $2$ singularities, counting with multiplicity.
This shows that through each point of $\PP^3 \setminus (C \cup V)$ there are $2$ secant or tangent lines to $C$.
A priori this is counting with multiplicity, but since $C$ has no trisecant or bitangent lines,
the two secant or tangent lines are distinct.
The point $x$ has rank $3$ if and only if no proper secant to $C$ passes through $x$.
Thus the points of rank $3$
are exactly those in the intersection of two tangent lines to $C$, other than the points in $V$
(this is one of the results in \cite[Theorem 28]{MR2736357}).

Let $W_3^{\circ}$ denote the set of points of rank $3$.
Every point in $W_3 = \overline{W_3^{\circ}}$ lies on at least two tangent lines of $C$,
by semicontinuity of the degree of the projection map from the abstract tangent variety
$\{(x,\ell) \mid x \in \ell, \text{ $\ell$ tangent to $C$}\}$.
But no point of $C$ lies on more than one tangent line.
This shows that $W_3$ is disjoint from $C$.

Let $Q$ be a smooth quadric containing $C$ and
let $\pr_i:Q\to\PP^1$, $i=1,2$, be the two natural projections.
Then $\pr_{i|C}:C\to\PP^1$ is a $2:1$ morphism with $4$ ramification points.
This shows that there are $4$ tangent lines to $C$ in each ruling.
The tangent lines to $C$ in the rulings of $Q$ intersect in $16$ points which do not lie on $C$,
as no tangent line intersects $C$ anywhere other than its point of tangency.
Therefore the $16$ points of intersection are in $W_3^{\circ}\cap Q$.
On the other hand, if $w\in W_3^{\circ}\cap Q$ then the tangent lines passing through $w$ are contained in $Q$.
Hence  any such quadric intersects $W_3^{\circ}$ in exactly 16 points.
This shows that the closure $\overline{W_3^{\circ}}$ is a curve of degree $8$.

Finally, let $Q$ be a singular quadric containing $C$.
Then it is immediate to realize that the vertex of $Q$ is the only point of $Q$
that lies on more than one tangent line of $C$.
So $W_3$ contains each vertex $x_1,\dotsc,x_4$ of a singular quadric through $C$.
These are the only points of rank $2$ in $W_3$.
\end{proof}

\begin{remark}
The above proof also recovers the (previously known) fact that
general points in $\PP^3$, namely those outside of the tangential variety of $C$,
admit precisely $2$ decompositions as linear combinations of $2$ points in $C$.
This holds more generally for elliptic normal curves of even degree, see \cite[Proposition 5.2]{MR2225496}.
\end{remark}

\begin{remark}
   The example of the elliptic quartic curve shows that $W_m$ can be disjoint from the base variety.
   Thus the situation as in the proof of Theorem~\ref{thm_dimension_bound_for_homogeneous}, when $W_m \supset X$, 
     is rather special to the homogeneous spaces.
\end{remark}

\subsection{General curves in a quadric surface}
\begin{proposition}
Let $a \geq b \geq 1$ and let $C \subset Q$ be a general curve of type $(a,b)$ in the smooth quadric surface $Q$.
If $a \geq 4$ then $W_3$ is empty, that is, the maximal rank $m$ is equal to $2$.
\end{proposition}

\begin{proof}
First let $x \in Q$, $x \notin C$.
Let $l$ be the line in $Q$ through $x$ such that $l \cdot C = a$.
By generality $l \cap C$ has points of multiplicity at most $2$,
so $a \geq 3$ is enough to imply that $l \cdot C$ is supported in at least two distinct points.
Hence $\rank(x) = 2$.

Next let $x \notin Q$,
and suppose $\rank(x) = 3$.
Let $\pi : \PP^3 \dasharrow \PP^2$ be the projection from $x$.
Since $x$ lies on no secant line to $C$, and not every tangent line to $C$ passes through $x$,
$\pi$ has degree $1$ on $C$.
Then $\pi(C)$ is a plane curve of degree $a+b$, and every point of $\pi(C)$ has multiplicity at most $2$,
since each line through $x$ intersects $Q$ with multiplicity $2$.
The projection $\pi(C)$ has no nodes, only cuspidal singularities.
Let $H_x$ be the polar hyperplane of $Q$ in $x$,
so $y \in H_x \cap Q$ if and only if the tangent plane to $Q$ at $y$ contains $x$,
see for example \cite[pg.~238]{harris:intro}, \cite[\S1.1.2]{MR2964027}.
Let $Z_x = H_x \cap C$.
Then $Z_x$ has degree $a+b$ and the cuspidal points of $\pi(C)$ are contained in $\pi(Z_x)$.
Therefore the curve $\pi(C)$ has at most $a+b$ cusps.
By adjunction in $Q$, $C$ has genus $1 + (1/2)a(b-2) + (1/2)b(a-2) = (a-1)(b-1)$.
The projection $\pi(C)$ has degree $a+b$, geometric genus $(a-1)(b-1)$, and only ordinary cusps,
hence the number of cusps is $(1/2)(a+b-1)(a+b-2)-(a-1)(b-1) = \binom{a}{2}+\binom{b}{2}$.
This is strictly greater than $a+b$ as soon as $a \geq 4$ and $b \geq 1$.
Thus once again $\rank(x) = 2$.
\end{proof}

\subsection*{Acknowledgements}

We are grateful for the hospitality and partial support of the Fields Institute in Toronto (Canada)
during the Thematic Program on Combinatorial Algebraic Geometry in Fall 2016.
We thank the program organizers and the Institute's staff for invitations, financial support,
and the wonderfully stimulative atmosphere for collaborative work.
We are also grateful to the participants of the semester for their suggestions and discussions.

J.~Buczy\'nski is supported by a grant of Polish National Science Center (NCN), project 2013/11/D/ST1/02580,
        and by a scholarship of Polish Ministry of Science.
K.~Han is supported by the POSCO Science Fellowship of POSCO TJ Park Foundation
and the DGIST Start-up Fund of the Ministry of Science, ICT and Future Planning (No.\ 2016010066).
M.~Mella is partially supported by Progetto MIUR ``Geometry of Algebraic Varieties'' 2015EYPTSB\_005.
Z.~Teitler is supported by a grant from the Simons Foundation (\#354574).

The computer algebra program Magma \cite{magma} was helpful in calculation of explicit examples.
The article is a part of the activities of the AGATES research group.

\renewcommand{\MR}[1]{}


\providecommand{\bysame}{\leavevmode\hbox to3em{\hrulefill}\thinspace}
\providecommand{\MR}{\relax\ifhmode\unskip\space\fi MR }
\providecommand{\MRhref}[2]{%
  \href{http://www.ams.org/mathscinet-getitem?mr=#1}{#2}
}
\providecommand{\href}[2]{#2}

\end{document}